\documentclass[11pt,english,oneside]{amsart}

\usepackage[breaklinks,colorlinks,linkcolor=blue,citecolor=blue,urlcolor=blue]{hyperref}
\usepackage[T1]{fontenc}
\allowdisplaybreaks[4]
\usepackage[latin9]{inputenc}
\usepackage{geometry}
\usepackage{indentfirst}
\usepackage{amssymb}
\usepackage{color}
\usepackage{graphicx}
\usepackage{subfigure}
\usepackage{enumerate}
\usepackage{float}

\usepackage{booktabs}
\usepackage{array, caption, threeparttable}
\usepackage[font=small,labelfont=bf,labelsep=none]{caption}
\usepackage{graphicx,color}
\usepackage{amsmath, amssymb, graphics}
\usepackage{graphicx}

\numberwithin{figure}{section}

\makeatletter
\numberwithin{equation}{section} %% Comment out for sequentially-numbered
\numberwithin{figure}{section} %% Comment out for sequentially-numbered
\@ifundefined{theoremstyle}{\usepackage{amsthm}}{}
\theoremstyle{plain}

\theoremstyle{plain}

\theoremstyle{plain}

\theoremstyle{remark}

\theoremstyle{remark}

\theoremstyle{plain}

\usepackage{geometry}

\newcommand{\R}{\mathbb{R}}

\smallskip
\def\<{{\langle }}
\def\>{{\rangle }}

\makeatother

\usepackage{babel}

\smallskip
\def\<{{\langle }}
\def\>{{\rangle }}

\makeatother

\usepackage[mathscr]{eucal}
\usepackage{amssymb}
\usepackage{latexsym}
\usepackage{amsthm}
\theoremstyle{plain}
\newtheorem{theorem}{Theorem}[section]
\newtheorem*{theorem*}{Theorem}
\newtheorem*{thm*}{Theorem}
\newtheorem{proposition}{Proposition}[section]
\newtheorem{corollary}{Corollary}[section]
\newtheorem{remark}{Remark}[section]
\newtheorem{lemma}{Lemma}[section]

\newtheorem{definition}{Definition}[section]

\makeatletter
\@addtoreset{equation}{section}

\usepackage{color}

\usepackage{graphicx,color}

\title[Embedded cylindrical and doughnut-shaped $\lambda$-hypersurfaces]
{Embedded cylindrical and doughnut-shaped $\lambda$-hypersurfaces}
\author{Qing-Ming Cheng,  Junqi Lai and Guoxin Wei}

\address{Qing-Ming Cheng \\ Department of Applied Mathematics, Faculty of Science,
Fukuoka  University, 814-0180, Fukuoka,  Japan, cheng@fukuoka-u.ac.jp}
\address{Junqi Lai \\  School of Mathematical Sciences, South China Normal University,
510631, Guangzhou,  China, 2019021668@m.scnu.edu.cn}
\address{Guoxin Wei \\  School of Mathematical Sciences, South China Normal University,
510631, Guangzhou,  China, weiguoxin@tsinghua.org.cn}

\begin{document}
	\maketitle
	
\begin{abstract}
In the paper, we construct, for $\lambda>0$,  complete embedded and non-convex $\lambda$-hypersurfaces, which
are diffeomorphic to a cylinder. Hence, one can not expect that $\lambda$-hypersurfaces share a common  conclusion
on the  planar domain conjecture  even if the planar domain conjecture of T. Ilmanen for self-shrinkers of mean
curvature flow are solved by Brendle \cite{B} affirmatively. Furthermore, for a fixed $\lambda<0$ which may have
small   $|\lambda|$,
we can  construct two  compact embedded $\lambda$-hypersurfaces
which are diffeomorphic to $\mathbb{S}^{1} \times \mathbb{S}^{n-1} $, but they  are not  isometric to each other.
\end{abstract}
\footnotetext{The first author was partially  supported by JSPS Grant-in-Aid for Scientific Research (C) No.22K03303.
The third author was partly supported by NSFC Grant No.12171164.}
	
\section{Introduction}
\noindent
A hypersurface $\Sigma^n \subset \R^{n+1}$ is called a $\lambda$-hypersurface if it satisfies
\begin{equation}\label{0316eq1.1}
H + \left\langle X, \nu \right\rangle= \lambda,
\end{equation}
where $\lambda$ is a constant, $X$,  $\nu$  and $H$  are  the position vector,  a unit
normal vector and the mean curvature of the hypersurface $\Sigma^n \subset \R^{n+1}$, respectively.
Note that the standard sphere with inward unit normal vector
has positive mean curvature in our convention.
The notation of $\lambda$-hypersurfaces were first introduced by Cheng and Wei in \cite{CW} (also see \cite{MR}).
Cheng and Wei \cite{CW} proved that $\lambda$-hypersurfaces are critical points of the weighted area functional with respect to
weighted volume-preserving variations.
$\lambda$-hypersurfaces can also be viewed as stationary solutions to the isoperimetric problem in the
Gaussian space. For more information on $\lambda$-hypersurfaces, one can see \cite{CW} and \cite{MR}.
\\
Since an orientable hypersurface has two directions of the unit normal vector,  if you change the direction of the normal vector,
the $\lambda$ will change its sign. In this paper, we choose the inward unit normal vector.
\\	
It  is well known that  there are several special complete embedded solutions to (\ref{0316eq1.1}):
\begin{itemize}
\item hyperplanes with a distance of $|\lambda|$ from the origin,
\item sphere with radius $\frac{-\lambda + \sqrt{\lambda^2 + 4n }}{2}$ centered at origin,
\item cylinders with an axis through the origin and radius $\frac{-\lambda + \sqrt{\lambda^2 + 4(n - 1) }}{2}$.
\end{itemize}

\begin{remark}
If $\lambda=0$, $\langle X, \nu\rangle +H=0$, then $X:\Sigma^n\to  \mathbb{R}^{n+1}$ is a self-shrinker of mean curvature flow,
that is,  self-shrinkers are  $0$-hypersurfaces. Hence,
one knows that  $\lambda$-hypersurfaces are  a natural generalization of  self-shrinkers
of  mean curvature flow, which play an important role for study on singularities of the mean curvature flow.
\end{remark}

\noindent
In \cite{CW1}, Cheng and Wei constructed the first nontrivial example of a $\lambda$-hypersurface
which is diffeomorphic to $\mathbb S^{n-1} \times \mathbb S^1$.
In \cite{R}, using a similar method to McGrath \cite{M}, Ross constructed a $\lambda$-hypersurface in $\R^{2n+2}$
which is diffeomorphic to $\mathbb S^{n} \times \mathbb S^{n} \times \mathbb S^1$ and exhibits a $SO(n) \times SO(n)$ rotational symmetry.

\noindent
In \cite{LW}, Li and Wei constructed an immersed  $\lambda$-hypersurface, which is homeomorphic  to $\mathbb S^n$.
For $0$-hypersurfaces (that is, self-shrinkers), Brendle \cite{B} proved that if $M^2$ is a properly embedded self-shrinker which is diffeomorphic to $\mathbb{R \times S}^{1} $, then $M$ is a cylinder. Brendle's result confirms the planar domain conjecture of T. Ilmanen \cite{I}. For  higher  dimensions, Kleene and M\o ller \cite{KM} obtained that the cylinder is  the only complete embedded, rotationally symmetric shrinkers of the topological type  $\mathbb{R \times S}^{n-1} $.
Kleene and M\o ller also proved that a smooth self-shrinkers of revolution, which is generated by rotating an entire graph around the $x$-axis must be the round cylinder $\mathbb{R \times S}^{n-1} $.
 \\
 In this paper, for $\lambda$-hypersurface ($\lambda\neq0$), motivated by \cite{D,DK,LW, S}, we construct nontrivial embedded $\lambda$-hypersurfaces,
  which are diffeomorphic to $\mathbb{R \times S}^{n-1} $. Hence,  the conclusion on the  planar domain conjecture for $\lambda$-hypersurfaces can
  not be expected.
  The examples that we construct in the theorem \ref{0316thm1.1} are also  rotational symmetric. Therefore,
 there exist  complete embedded cylindrical type $\lambda$-hypersurfaces besides the round cylinder.
 Moreover, we know from the theorem \ref{0316thm1.1} that there  also  exist entire graph $\lambda$-hypersurfaces besides the round cylinder.

\begin{theorem}\label{0316thm1.1}
Given an integer $n \ge 2$, there is a $c_1(n)>0$ only depending on $n$ such that, for $0<\lambda<c_1(n)$,
there exists a complete  embedded non-convex $\lambda$-hypersurface $\Sigma^n \subset \R^{n+1}$
which is diffeomorphic to $\mathbb{R \times S}^{n-1} $ and is not  isometric to the  round cylinder.
\end{theorem}

\begin{remark}
In \cite{Hei}, Heilman proved that, for $\lambda>0$, convex $n$-dimensional properly $\lambda$-hypersurfaces are isometric to  the standard sphere,
cylinder, the hyperplane.  Heilman's theorem generalized the rigidity result of Colding and Minicozzi \cite{CM}.
We need to notice that   examples constructed in the theorem \ref{0316thm1.1} are non-convex.
\end{remark}
\noindent
For $0$-hypersurfaces (that is, self-shrinkers), it is conjectured \cite{CS,DK,KM,R1} that there are no embedded rotationally
invariant self-shrinkers of topological
type $\mathbb{S}^1 \times \mathbb{S}^{n-1}$ in $\mathbb{R}^{n+1}$ other than Angenent's example constructed in \cite{A}.
But for a fixed $\lambda<0$ nearing zero, motivated by \cite{D,DK,LW, S},
we can construct two $\lambda$-hypersurfaces diffeomorphic to $\mathbb{S}^1 \times \mathbb{S}^{n-1}$, but they are not isometric to each other.
In fact, we prove

\begin{theorem}\label{0316thm1.2}
Given an integer $n \ge 2$, there is a $c_2(n)>0$ only depending on $n$ such that, for a fixed $\lambda$  satisfying
$-c_2(n)<\lambda<0$,
there exists two  embedded $\lambda$-hypersurfaces $\Sigma_{1}^n,\ \Sigma_{2}^n \subset \R^{n+1}$
which are diffeomorphic to $\mathbb{S}^{1} \times \mathbb{S}^{n-1} $, but they are not isometric to each other.
\end{theorem}

\begin{remark}
The $\lambda$-hypersurfaces we construct in theorem \ref{0316thm1.2} are rotationally invariant.
So one can not expect that  $\lambda$-hypersurfaces with  $\lambda<0$  share the common features
with self-shrinkers.\\
But for $\lambda>0$, by numerical evidence, we propose that  there are no embedded rotationally invariant
$\lambda$-hypersurfaces of topological type $\mathbb{S}^1 \times \mathbb{S}^{n-1}$ in $\mathbb{R}^{n+1}$
other than  the examples were constructed by Cheng and Wei in \cite{CW1}.\\
Moreover, it is worth note that a non-rotationally, embedded, genus one self-shrinker in $\mathbb{R}^3$
was constructed by Chu and Sun in \cite{CS}. \\
In \cite{CLW}, Cheng, Lai and Wei proved that, for  $n \ge 2$ and $ -\frac{2}{\sqrt{n+2}} < \lambda <0$,
there exists an  embedded convex $\lambda$-hypersurface $\Sigma^n \subset \R^{n+1}$
which is diffeomorphic to $\mathbb{S}^n$ and is not  isometric to the  standard sphere.
\end{remark}

\section{Preliminaries}\label{prel}

\noindent
Let $\mathbb{H} = \left\lbrace (x,r)\in\R^2: x\in\R, \  r > 0\right\rbrace$ and $SO(n)$ denote the special orthogonal group
and act on $\R^{n+1} = \left\lbrace (x,y): x\in\R, \ y\in\R^n\right\rbrace $ in the usual way.
We can identify the space of orbits $\R^{n+1}/SO(n)$
with the half plane $\overline{\mathbb{H}} = \left\lbrace (x,r)\in\R^2: x\in\R, \ r \ge 0\right\rbrace$
under the projection(see \cite{R})
	\[
	\Pi(x,y) = (x,|y|) = (x,r).
	\]
If a hypersurface $\Sigma$ is invariant under the action $SO(n)$, then the projection $\Pi(\Sigma)$ will
give us a profile curve in the half plane, which can be parametrized by Euclidean arc length
and write as $\gamma(s) = (x(s), r(s))$. Conversely, if we have a curve $\gamma(s) = (x(s), r(s)),\ \ s \in (a, b)$
parametrized by Euclidean arc length in the half plane,  we can reconstruct the hypersurface by
\begin{equation}\label{0316eq2.1}
\begin{aligned}
&{ X : (a, b) \times S^{n-1}(1) \to \R^{n+1} }, \\
&(s,\alpha) \mapsto (x(s), r(s)\alpha).
\end{aligned}
\end{equation}
Let
\begin{equation}\label{0316eq2.2}
\nu = (-\dot{r}, \dot{x}\,\alpha),
\end{equation}
where the dot denotes the  derivative with respect to arc length $s$.
A direct calculation shows that $\nu$ is a unit normal vector for the hypersurface.
Then we can calculate that the principal curvatures of the hypersurface (see \cite{CW1,DD}):
\begin{equation}\label{0426eq2.3}
\begin{aligned}
\kappa_i & = - \frac{\dot{x}}{r}, \ \ \ \ i=1,\,2,\,\dots,\,n-1, \\
\kappa_n &= \dot{x}\,\ddot{r}-\ddot{x}\,\dot{r}.
\end{aligned}
\end{equation}
Hence the mean curvature $H$ satisfies
\begin{equation}\label{0316eq2.3}
H =  \dot{x}\,\ddot{r}-\ddot{x}\,\dot{r} - (n-1)\frac{\dot{x}}{r}
\end{equation}
and  by (\ref{0316eq2.1}), (\ref{0316eq2.2}) and (\ref{0316eq2.3}),
equation (\ref{0316eq1.1}) reduces to (see also \cite{CW1,DK,R})
\begin{equation}\label{0316eq2.4}
\dot{x}\,\ddot{r} - \ddot{x}\,\dot{r} = (\frac{n-1}{r} - r )\dot{x} +   x\,\dot{r} + \lambda,
\end{equation}
where $(\dot{x})^2 + (\dot{r})^2 = 1$. \\
Let $\theta(s)$ denote the angle between the tangent vector of the profile curve
and $x$-axis.  (\ref{0316eq2.4}) can be written as the following system of differential equation:
\begin{equation}\label{0316eq2.5}
\left\lbrace
\begin{aligned}
\dot{x} &= \cos\theta,\\
\dot{r} &= \sin\theta, \\
\dot{\theta} &=  (\frac{n-1}{r} - r )\cos\theta +   x\,\sin\theta + \lambda.
\end{aligned}
\right.
\end{equation}
Let $P$ denote the projection from $\overline{\mathbb{H}} \times \R$ to $\overline{\mathbb{H}}$. Obviously, the curve $ P(\Gamma(s))$ generates a $\lambda$-hypersurface via (\ref{0316eq2.1}) provided $\Gamma(s)$ is a solution of (\ref{0316eq2.5}).
\\
Letting $(x_0, r_0, \theta_0)$ be a point in ${\mathbb{H}} \times \R$, by the  existence and uniqueness theorem of the solutions for first order ordinary differential equations, there is a unique solution $\Gamma(x_0, r_0, \theta_0)(s)$ to (\ref{0316eq2.5}) satisfying initial conditions $\Gamma(x_0, r_0, \theta_0)(0) = (x_0, r_0, \theta_0) $. Moreover, the solution depends smoothly on the initial conditions. For the convenience, we henceforth denote $\Gamma(0, \delta, 0)(s)$ by $\Gamma_\delta(s)$, $P(\Gamma_\delta(s))$ by $\gamma_\delta(s)$ and assume $\Gamma_\delta(s) = (x_\delta(s), r_\delta(s), \theta_\delta(s))$. If the curve $\gamma_\delta(s)$ is simple and exits the upper-half plane through infinity, then the curve $\gamma_\delta(s)$ will generate a cylindrical embedded $\lambda$-hypersurface via (\ref{0316eq2.1}). Similarly, a smooth simple closed  profile curve in $\mathbb{H}$ generates an embedded $\lambda$-hypersurface which is diffeomorphic to $\mathbb{S}^{1} \times \mathbb{S}^{n-1} $. This paper's main purpose  is to find such curves.\\
When the profile curve can be written in the form $(x, u(x))$, by (\ref{0316eq2.5}), the function $u(x)$ satisfies the differential equation
\begin{equation}\label{0528eq2.7}
\frac{u''}{1 + (u')^2} = x\,u' - u + \frac{n-1}{u} + \lambda\sqrt{1+(u')^2}.
\end{equation}
It is obvious that
$u = \frac{\lambda + \sqrt{\lambda^2 + 4(n - 1) }}{2}$
is a solution of  (\ref{0528eq2.7}). We will denote $\frac{\lambda + \sqrt{\lambda^2 + 4(n - 1) }}{2}$ by $R_\lambda$. 	
\\
When the profile curve can be written in the form $(f(r), r)$, by (\ref{0316eq2.5}), the function $f(r)$ satisfies the differential equation
\begin{equation}\label{0526eq2.8}
\frac{f''}{1 + (f')^2} =  ( r - \frac{n-1}{r}  )f' -f - \lambda\sqrt{1+(f')^2}.
\end{equation}
Differentiating  the above equation with respect to $r$, we obtain
\begin{equation}\label{0526eq2.9}
\frac{f'''}{1 + (f')^2} = \frac{2f'(f'')^2}{(1 + (f')^2)^2}+ ( r - \frac{n-1}{r}  )f'' + \frac{n-1}{r^2}f' -\frac{\lambda f'f''}{\sqrt{1 + (f')^2}}.
\end{equation}
Note that (\ref{0526eq2.8}) has a solution $f = -\lambda$, which corresponds to hyperplane, we will refer this constant solution as the plane.
We conclude this section with a lemma about the solutions of (\ref{0526eq2.8}),
which can be obtained easily by (\ref{0526eq2.8}), (\ref{0526eq2.9}) and the  uniqueness theorem for equation (\ref{0526eq2.8}).
\begin{lemma}\label{0527lem2.1}
Let  $f: (a,b) \to \mathbb{R}$ be a solution of $(\ref{0526eq2.8})$, $c \in (a,b)$ with $f''(c) = 0$. If $f'(c) = 0$ or $f'''(c) = 0$,
then $f=-\lambda$. Moreover, $f'(c)f'''(c)>0$ provided $f$ is not the plane.
\end{lemma}
	
\section{A classification of $\delta$ in $(0,R_\lambda)$}
	
\noindent
In this section, we will give some descriptions of the behavior of the curve $\gamma_\delta$ for $0< \delta <  R_\lambda$,
and a classification of $\delta$. We begin by studying the behavior of general solutions for the equation (\ref{0526eq2.8}).
\begin{lemma}\label{0527lem3.1}
Let  $f: (a,b) \to \mathbb{R}$ be a solution of $(${\rm\ref{0526eq2.8}}$)$, $c \in (a,b)$, we have the following assertions:
\begin{enumerate}[$(1)$.]
\item If $f'(c) f''(c) <0$, then $f'(r)f''(r) <0$ for all $r \in (a, c]$,
\item If $f'(c)f''(c) > 0$, then $f'(r)f''(r) >0$ for all $r \in [c, b)$.
\end{enumerate}
\end{lemma}
\begin{proof}
We only prove the first part of the lemma. The second part can be proved similarly. If $f'(c) > 0$ and $ f''(c) <0$, by  the continuity of $f''(r)$,
we know that $f''(r) <0$ when $r$ near $c$. Hence, if $f''(r) \ge 0$ for some $r \in (a, c]$, then there must be a point $\bar{c} \in (a,c)$
such that $f''(\bar{c}) = 0$ and $f''(r) < 0$ for all $r \in (\bar{c}, c]$. The choice of $\bar{c}$ implies that $f'(\bar{c}) > f'(c) > 0$ and $f'''(\bar{c}) \le 0$,
which contradicts lemma \ref{0527lem2.1}.  Hence, we conclude  $f''(r) <0$ for all $r \in (a, c]$
which yields that $f'(r) > 0$ and $f'(r)f''(r) <0$ for all $r \in (a, c]$.
If $f'(c) < 0$ and $f''(c) >0$, the proof is similar to the above.
%Consider (\ref{0526eq2.9}) as an equation for $f'$, one can say that it satisfies a maximum principle, and therefore our assertions are obvious.
\end{proof}
 \noindent
 This lemma together with lemma \ref{0527lem2.1} show the following two corollaries.
\begin{corollary}\label{0528cor3.1}
Let $f$ be a solution of $(${\rm \ref{0526eq2.8}}$)$,  which is not the plane. If $f' = 0$ at a point, then $f$ is either strictly convex or strictly concave.
\end{corollary}
\begin{corollary}\label{0528cor3.2}
Let $f: (a, b) \to \mathbb{R}$ be a solution of $(${\rm \ref{0526eq2.8}}$)$, which is not the plane and at some point  $c \in (a,b)$, $f''(c) =0$.
If $f'(c) > 0$, then $f''(r) < 0$ for $r \in (a,c)$, $f''(r) > 0$ for $r \in (c,b)$; if $f'(c) < 0$, then $f''(r) > 0$
for $r \in (a,c)$, $f''(r) < 0$ for $r \in (c,b)$. Hence, there exists  at most one point in $(a, b)$ such that $f''=0$.
\end{corollary}
\noindent
We also need the following lemma in order to fix some interval, as we shall see in the proof of proposition \ref{0529lem3.7}.
\begin{lemma}\label{230527lem3.2}
Let $f: (a, b) \to \mathbb{R}$ be a right maximally extended solution of $(\ref{0526eq2.8})$.
If $f' > 0$ and $f''> 0$ on $(a, b)$, then $b > R_\lambda$. If $f' < 0$ and $f''< 0$ on $(a, b)$, then $b > R_{-\lambda}$.
\end{lemma}
\begin{proof}
We only consider the case of $\lambda <0$, the proof of the case of $\lambda \ge 0$ is similar.
Suppose $f' > 0$, $f''> 0$ on $(a, b)$ and $b \le R_\lambda$.
We have  $\lim_{r \rightarrow b}f(r) = \infty$ or $\lim_{r \rightarrow b}f'(r) = \infty$
since $f$ is a right maximally extended  solution,  that is, $f$ blows-up at $b$.
If $b < R_\lambda$, by (\ref{0526eq2.8}) we have
\begin{align}
\frac{f''}{1 + (f')^2}
&=  ( r - \frac{n-1}{r}  )f' -f - \lambda\sqrt{1+(f')^2}\nonumber\\
&<   ( r - \frac{n-1}{r}  )f' -f - \lambda (1+f') \nonumber\\
&=   ( r - \frac{n-1}{r} -\lambda )f' -f - \lambda,\nonumber
\end{align}
which implies that $f'' < 0$ as $r \rightarrow b$. This contradicts $f''>0$ on $(a,b)$.
\\
If  $b = R_\lambda$, the above inequality  implies that
\[f< ( r - \frac{n-1}{r} - \lambda )f' - \frac{f''}{1 + (f')^2} - \lambda\\
<   - \lambda\]
for $r \in (a,b)$. Therefore there exists $x_* \le -\lambda $ such that $\lim_{r \rightarrow b}f(r) = x_*$.
Near the point $(x_*, b)$, we write the curve $(f(r), r)$ as $(x,u(x))$,
where $u$ satisfies the differential equation (\ref{0528eq2.7}). Now, $u(x_*) = R_\lambda$ and $u'(x_*) = 0$,
and by the uniqueness of solutions for equation (\ref{0528eq2.7}), $u$ must be the constant function $u(x) = R_\lambda$.
This contradicts that $(f(r), r)$ agrees with $(x,u(x))$ near $(x_*,b)$. Hence, we obtain $b > R_\lambda$.
The second part of the lemma can be proved similarly.
\end{proof}
\noindent
Now we come back to consider the profile curves  $\gamma_\delta(s) = P(\Gamma_\delta(s))$.
We indicate some simple facts on $\Gamma_\delta(s)$:
\begin{itemize}
\item Since $\theta_\delta(0)=0$, we have $\theta_\delta \ne \pi/2,\pi$ in a small neighborhood of $s=0$.
\item From  $\dot{\theta}_\delta(0) = -(\delta^2 - \lambda\,\delta - (n-1))/\delta$, we know $\dot{\theta}_\delta(0) \ne 0$
and $\dot{\theta}_\delta \ne 0$ in a small neighborhood of $s=0$ provided $\delta \ne R_\lambda$.
In particular, $\dot{\theta}_\delta> 0$ in a small neighborhood of $s=0$ provided $0<\delta < R_\lambda$.
\item If $\delta \ne R_\lambda$, then $\theta_\delta \ne 0$ in a chosen  neighborhood of $s=0$ by the above item.
In particular, $\theta_\delta(s) > 0$ for small $s>0$ provided $0<\delta < R_\lambda$.
\item Because of $x_\delta(0) = 0$ and $\dot{x}_\delta(0)=\cos(\theta_\delta(0))=1$, we have  $x_\delta \ne 0$ in a chosen neighborhood of $s=0$.
In particular, $x_\delta(s) > 0$ for small $s>0$.
\end{itemize}
Hence in the case of $\delta \ne R_\lambda$, the following
definitions of $S(\delta)$, $s_1(\delta)$, $s_2(\delta)$, $s_3(\delta)$ and $s_4(\delta)$
are reasonable (see \cite{A,CW1}). Henceforth we assume $0<\delta < R_\lambda$.
\begin{definition}\label{231116def3.1}
For $0<\delta < R_\lambda$, we define:
\begin{enumerate}[$(1)$.]
\item Let $ S(\delta) > 0$ be the real number such that $\Gamma_\delta(s) : [0, S(\delta)) \to \mathbb{H} \times \R$
is the right maximally extended solution of the system $(\ref{0316eq2.5})$.
\item Let $s_1(\delta) > 0$ be the arc length of the first time, if any, at which
either  $\theta_\delta = 0$ or $\theta_\delta = \pi$. If these never happen, we take $s_1(\delta) = S(\delta)$.
\item Let $ s_2(\delta)>0$ be the arc length of the first time, if any, at which
$\dot{\theta}_\delta = 0$. If this never happen, we take ${s}_2(\delta) = S(\delta)$.
\item Let $ s_3(\delta)>0$ be the arc length of the first time, if any, at which
either $\theta_\delta = 0$ or $\theta_\delta = \pi/2$. If these never happen, we take ${s}_3(\delta) = S(\delta)$.
\item Let $ s_4(\delta)>0$ be the arc length of the first time, if any, at which
 $x_\delta = 0$ . If this never happen, we take ${s}_4(\delta) = S(\delta)$.
\end{enumerate}
\end{definition}
\noindent
For brevity, we will  denote $s_1(\delta)$, $s_2(\delta)$, $s_3(\delta)$, $s_4(\delta)$, $S(\delta)$
by  $s_1$, $s_2$, $s_3$, $s_4$, $S$, respectively.
By the definitions, we know
$$
\begin{aligned}
&0< \theta_\delta < \pi\  {\rm in} \  (0,s_1),\\
&\dot{\theta}_\delta > 0  \ {\rm in}\  (0,s_2),\\
& 0< \theta_\delta < \pi/2 \ {\rm in}  (0,s_3),\\
& x_\delta >0\ {\rm in}  (0,s_4).
\end{aligned}
   $$
Therefore $\dot{r}_\delta = \sin \theta_\delta > 0$ in $(0, s_1)$ and  $\lim_{s \to s_1}r_\delta(s)>0$
always exists although it might be $\infty$.
As we shall see in the following proposition 3.1, $\lim_{s \to s_1}x_\delta(s)$ and $\lim_{s \to s_1}\theta_\delta(s)$
also always exist due to $\dot{x}_\delta(s)$ and $\dot{\theta}_\delta(s)$ can not be zero for $s$ close to $s_1$ from below.\\
Since $\Gamma_\delta(s)$ is continuous on $s$,
we may write $\lim_{s \to s_1}x_\delta(s)$, $\lim_{s \to s_1}r_\delta(s)$, $\lim_{s \to s_1}\theta_\delta(s)$
into $x_\delta(s_1)$, $r_\delta(s_1)$, $\theta_\delta(s_1)$, respectively, without confusion.
We may also denote $r_\delta(s_1)$ by $b_\delta$.
\\
In addition, the curve $\gamma_\delta(s)$, $0<s<s_1$ can be written as a graph over $r$-axis since $\dot{r}_\delta>0$.
We denote this graph by $(f_\delta(r), r)$ where $\delta < r < r_\delta(s_1) $. Thus, the  function $f_\delta(r),\ \delta<r< r_\delta(s_1)$ is a maximally extended solution of (\ref{0526eq2.8}) by definition of $s_1$. \\
It is worth noting that, for $s_0$,
\begin{equation}\label{eq3.1}
\begin{aligned}
&\theta_\delta(s_0)=\pi/2, \ 0<\theta_\delta(s_0)<\pi/2, \  \pi/2<\theta_\delta(s_0)<\pi, \\
&\dot{\theta}_\delta(s_0)=0, \
\dot{\theta}_\delta(s_0)>0, \ \dot{\theta}_\delta(s_0)<0
\end{aligned}
\end{equation}
imply
\begin{equation}\label{eq3.2}
\begin{aligned}
&f'_\delta(r_\delta(s_0)) =0, \ f'_\delta(r_\delta(s_0)) >0, \ f'_\delta(r_\delta(s_0)) <0, \\
& f''_\delta(r_\delta(s_0)) =0,
 \ f''_\delta(r_\delta(s_0)) <0, \ f''_\delta(r_\delta(s_0)) >0
 \end{aligned}
\end{equation}
 respectively, and vice versa.
Thus $f'_\delta(r) >0$, $f''_\delta(r) <0$ for $r$ close to $\delta$ because of  $\theta_\delta(0) =0$ and $\dot{\theta}_\delta(0)>0$.
\\
In light of the  corollary \ref{0528cor3.1} and the corollary \ref{0528cor3.2},
we have the following proposition \ref{231121lem3.3}, which describes the behaviors of the curve
$\gamma_\delta$ for $0< \delta <  R_\lambda$ and we give definitions of types on  $\delta$ in $(0,R_\lambda)$.

\begin{proposition}\label{231121lem3.3}
For $0<\delta < R_\lambda$, the following holds.
\begin{enumerate}[$(1)$.]
\item $s_1<S$ yields $\theta_\delta(s_1) =0$ or $\theta_\delta(s_1) =\pi$, where the former implies $s_2<s_1$ and the latter implies $s_3<s_1$.
\item $s_3 \le s_4$ and $s_3=s_4$ if and only if   $s_1=s_3=s_4=S$.
\item If $s_3<s_1$, we have   $s_1\le s_2$, $\pi/2<\theta_\delta(s)<\pi$ in $(s_3,s_1)$ and
 $r_\delta(s_1)>R_{-\lambda}$.\\
 If $s_3\leq s_4<s_1$, we have $x_\delta(s)<0$ in $(s_4,s_1)$.\\
If $s_3<s_1\le s_4$,   $s_1<s_2$ holds.
\item If $s_2<s_1$, we have $s_1=s_3$, $\dot{\theta}_\delta(s) <0$ in $(s_2,s_1)$, $r_\delta(s_1)>R_{\lambda}$.
\item If $s_3\ge s_1$ and $s_2\ge s_1$, we conclude  $s_1=s_2=s_3=s_4=S = \infty$ and $r_\delta(s_1) = \infty$.			
\end{enumerate}
\end{proposition}
\begin{proof}

(1). Assume $s_1<S$, then $\theta_\delta(s_1)=0$ or $\theta_\delta(s_1)=\pi$ by the definition.
If $\theta_\delta(s_1)=0$, then $\dot{\theta}_\delta(\xi_1)=0$ for some $\xi_1$ in $(0,s_1)$
because of  $\theta_\delta(0)=0$. Hence, $s_2\le\xi_1<s_1$.
If $\theta_\delta(s_1)=\pi$, then ${\theta}_\delta(\xi_2)=\pi/2$ for some $\xi_2$ in $(0,s_1)$
from $\theta_\delta(0)=0$, which yields
$s_3\le\xi_2<s_1$.
\\
 (2). By the definition, $0<\theta_\delta<\pi/2$ in $(0,s_3)$.  $x_\delta(s)= \int_{0}^{s}\cos{\theta}_\delta(\xi){\rm d}\xi >0$ for $s$ in $(0,s_3)$
 and therefore $s_3\le s_4$. If $s_3 = s_4$, it suffices to prove $s_3 =S$ according to  $s_3 \le s_1$.
If  $s_3 = s_4 <S$,  $x_\delta(s_4) =0$  yields that $\dot{x}_\delta(s_0) =0$ for some $s_0$ in $(0,s_3)$
from  $x_\delta(0)=0$,
which contradicts  $\dot{x}_\delta =\cos\theta_\delta>0$ in $(0,s_3)$ by the definition.
\\	
(3). If  $s_3<s_1$, we have  $\theta_\delta(s_3) =\pi/2$ by the definition.
Hence $f'_\delta(r_\delta(s_3)) =0$ and  $f''_\delta(r) >0$ or $f''_\delta(r) <0$ for
 $r \in (\delta,r_\delta(s_1))$ by the corollary \ref{0528cor3.1}.
Since $f''_\delta(r) <0$ when $r$ closes to $\delta$,  we have $f''_\delta(r) <0$ in $(\delta,r_\delta(s_1))$. 	
It follows that $\dot{\theta}_\delta>0$ in $(0,s_1)$, i.e., $s_1\le s_2$. By combining with $\theta_\delta(s_3)=\pi/2$
we get $\pi/2<\theta_\delta(s)<\pi$ in $(s_3,s_1)$.
Hence $f'_\delta(r) <0$, $f''_\delta(r) <0$ for $r\in (r_\delta(s_3),r_\delta(s_1))$
and  $r_\delta(s_1)>R_{-\lambda}$ by the lemma \ref{230527lem3.2}.
\vskip1mm
\noindent
If $s_4<s_1$, $x_\delta(s_4)=0$ by the definition. Therefore $x_\delta(s)<0$ in $(s_4,s_1)$
since $\dot{x}_\delta = \cos\theta_\delta<0$ in $(s_3,s_1)$ and $s_3\le s_4$ from  (2).
\vskip1mm
\noindent
If $s_1\le s_4$,  we know that  $x_\delta>0$ in $(0,s_1)$ which follows that $x_\delta(s_1)$ is actually finite by the
monotone bounded convergence theorem.
Moreover, $\theta_\delta(s_1)$ and $r_\delta(s_1)$ are also finite by the same theorem.
We point out that the upper bound for $r_\delta$ in $(0,s_1)$ comes from $s_1\le s_2$ and the lower bound for $x_\delta$ in $(0,s_1)$.
Now $(x_\delta(s_1),r_\delta(s_1),\theta_\delta(s_1))$ is in ${\mathbb{H}} \times \mathbb{R}$ and then $s_1<S$.
\\
Secondly, we prove $s_1<s_2$ by a contradiction.
Supposing $s_1 =s_2<S$, then $\dot{\theta}_\delta(s_1)=0$ by the definition
and $\theta_\delta(s_1) =\pi$ by (1).
But according to the existence and uniqueness of the system (\ref{0316eq2.5}),
the solution $(-s,R_{-\lambda},\pi)$ preserves  $\dot{\theta}_\delta(s)=0$ and $\theta_\delta(s) =\pi$ at both  $s_1$ and $s_2$.
This is a  contradiction.
\\
(4). From  $s_2<s_1$, we have  $s_1=s_3$ by (3).
Since $s_2<s_1$ then $\dot{\theta}_\delta(s_2)=0$ by the definition.
Together with $s_1=s_3$, this follows that $f'_\delta(r_\delta(s_2)) >0$ and $f''_\delta(r_\delta(s_2)) =0$,
which give that $f''_\delta(r)>0$ in $(r_\delta(s_2),r_\delta(s_1))$ by the corollary \ref{0528cor3.2}.
Hence $\dot{\theta}_\delta(s) <0$ in $(s_2,s_1)$.
$s_1=s_3$ gives that $f'_\delta(r)>0$ in $(r_\delta(s_2),r_\delta(s_1))$ from (\ref{eq3.1}) and (\ref{eq3.2}).
By combing with $f''_\delta(r)>0$ in the same interval,  we obtain $r_\delta(s_1)>R_\lambda$ by the lemma \ref{230527lem3.2}.
\\
(5). If $s_3\ge s_1$ and $s_2\ge s_1$ then $s_1=s_2=s_3=s_4=S$ by (1) and (2).  Therefore $S=\infty$ holds by the definition of $S$.
Otherwise,  if  $S<\infty$,  we have $x_\delta(s_1)$, $r_\delta(s_1)>0$ and $\theta_\delta(s_1)$ are all finite.  The solution can  be extended again.
This is a contradiction. \\
For $T$ with $0<T<s_1$,  we know  $\dot{r}_\delta \ge \sin\theta_\delta(T)>0$ in $[T,s_1)$ by $s_1=s_2=s_3$.
This yields $r_\delta(s_1)=\infty$ since $s_1 =\infty$.
\end{proof}
\begin{definition}\label{231117def3.2}
For $0<\delta < R_\lambda$,
\begin{enumerate}[ $(1)$]
\item $\delta$ is called  type 1 if $s_3(\delta)<s_1(\delta)$.
\item $\delta$ is called type 2 if $s_2(\delta)<s_1(\delta)$\,$($see figure \ref{fig3.1}\subref{fig:3.1a}$)$.
\item $\delta$ is called  type 3 if $s_3(\delta)=s_2(\delta)=s_1(\delta)$\,$($see figure \ref{fig3.1}\subref{fig:3.1b}$)$.
\item $\delta$ is called  type 1.1 if $s_3(\delta)<s_1(\delta)$ and $s_4(\delta)<s_1(\delta)$\,$($see figure \ref{fig3.1}\subref{fig:3.1c}$)$.
\item $\delta$ is called  type 1.2 if $s_3(\delta)<s_1(\delta)$ and $s_4(\delta)=s_1(\delta)$\,$($see figure \ref{fig3.1}\subref{fig:3.1d}$)$.
\item $\delta$ is called  type 1.3 if $s_3(\delta)<s_1(\delta)$ and $s_4(\delta)>s_1(\delta)$\,$($see figure \ref{fig3.1}\subref{fig:3.1e}$)$.
\end{enumerate}
\end{definition}

	\noindent
	\begin{figure}[H]\centering
		\subfigure[] {	
			\includegraphics[width=0.11\linewidth]{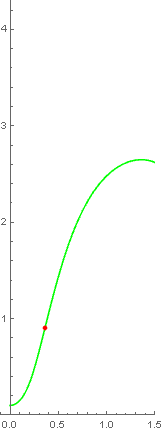}  \label{fig:3.1a}
		}\hfil
		\subfigure[] {
			\includegraphics[width=0.11\linewidth]{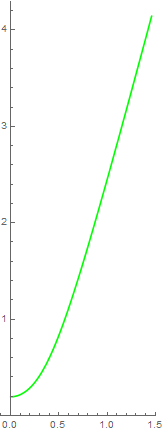}   \label{fig:3.1b}
		} \hfil
		\subfigure[] {
			\includegraphics[width=0.11\linewidth]{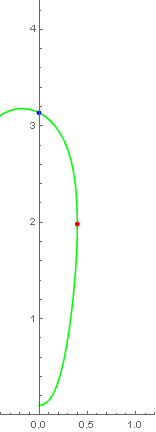}     \label{fig:3.1c}
		}  \hfil
	    \subfigure[] {
	        \includegraphics[width=0.11\linewidth]{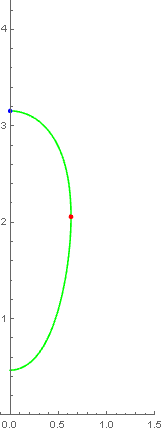}     \label{fig:3.1d}
        }  \hfil
        \subfigure[] {
            \includegraphics[width=0.11\linewidth]{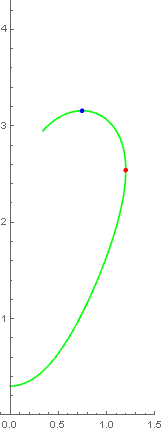}     \label{fig:3.1e}
        }
		\caption{\ Types of the profile curves.}
		\label{fig3.1}
	\end{figure}

\noindent
We conclude this section by noting that if $\delta$ is the  type 1.2, $x_\delta(s_1) =0$ and $\theta_\delta(s_1)=\pi$ by the definitions of $s_1$ and $s_4$.
By the  symmetry of system (\ref{0316eq2.5}), these two equalities
ensure that $\gamma_\delta(s)$, $0\le s\le 2s_1$ is a smooth simple closed profile curve,
which will generate a $\lambda$-hypersurface diffeomorphic to $\mathbb{S}^{1} \times \mathbb{S}^{n-1} $.
Also, if $\delta$ is  the  type 3, then $\gamma_\delta(s)$, $-\infty< s < \infty$ is simple
and exits upper-half plane through infinity. Thus it will generate
an embedded complete $\lambda$-hypersurface diffeomorphic to $\mathbb{R} \times \mathbb{S}^{n-1} $.

\section{Behavior of the curve $\gamma_\delta$ when perturbing $\delta$}

\noindent
In this section, the behavior of the curve $\gamma_\delta$ when disturbing $\delta$ for
each type or $\delta$ is near $0$ are studied.
In particular, for the situations of near $0$ and the type 3, the geometric  intuition is useful.
We introduce the following variables (see \cite{A,CW1}):
\begin{equation}\left\lbrace
	\begin{aligned}
	\xi(t) &= \frac{1}{\delta}x(\delta\,t),\nonumber\\
	\rho(t) &= \frac{1}{\delta}(r(\delta\,t) - \delta),\nonumber\\
	\alpha(t) &= \theta(\delta\,t),\nonumber
	\end{aligned}\right.
	\end{equation}
	where $\delta > 0$. From (\ref{0316eq2.5}), they satisfy
	\begin{equation}\label{0528eq3.1}
	\left\lbrace
	\begin{aligned}
	\xi' &= \cos \alpha,\\
	\rho' &= \sin \alpha,\\
	\alpha' &= \frac{n - 1}{1 + \rho}\xi' + \lambda\,\delta + \delta^2(\xi\,\rho' -(1+\rho)\xi').
	\end{aligned}
	\right.
\end{equation}
	Consider the system (\ref{0528eq3.1}) with initial conditions:
	\begin{equation}\label{0528eq3.2}
	\xi(0) = 0,\ \ \rho(0) = 0,\ \ \alpha(0) = 0.
	\end{equation}
For $\delta = 0$, this system can be solved explicitly, and one gets that $\alpha(t) = \arcsin \rho'(t)$ ,
where $\rho(t)$ is the inverse function of $$t = \int_{0}^{\rho}\frac{(1+\rho)^{n-1}}{\sqrt{(1 + \rho)^{2(n-1)} - 1}}{\rm d}\rho.$$
Since the solution of (\ref{0528eq3.1}) depends smoothly on the parameter $\delta$, we may conclude the following.
\begin{lemma}
For any $m \in \mathbb{N}^*$, there is a \ $T_m > 0$ and a \ $\delta_m > 0$ such that for all $0< \delta < \delta_m$,
one has $T_m\,\delta < s_2(\delta)$, and at $s = T_m\,\delta$,
 $\arctan m < \theta_\delta < \pi/2 $,  $x_\delta = O(\delta)$ and $r_\delta = \delta + O(\delta)$. Where $\mathbb{N}^*$ denotes
 the set of positive natural numbers.
\end{lemma}

\begin{remark}\label{0529rmk3.1}
For any fixed $m \in \mathbb{N}^*$, one can choose sufficiently small $\delta_m >0$
such that $x_\delta(T_m\,\delta) = O(\delta) <\frac{1}{m}$ and $r_\delta(T_m\,\delta) = \delta + O(\delta) <1$ for $0<\delta < \delta_m$.	
Consider the tangent line $l_{m,\delta} : r - r_\delta(T_m\,\delta)
= \tan \theta_\delta(T_m\,\delta)(x - x_\delta(T_m\,\delta))$ of $\gamma_\delta$ at $s = T_m\,\delta$.
Letting  $r = 1$ in the equation of $l_{m,\delta}$,
we get $$ x = \frac{1}{\tan \theta_\delta(T_m\,\delta)}(1 - O(\delta) - \delta) + O(\delta) < \frac{2}{m} $$ for $0<\delta < \delta_m$.
Therefore, if  $0< \delta < \delta_m$
and the profile curve $\gamma_\delta(s) = (x_\delta, r_\delta),\ s\in (0, b]$
satisfies  $\dot{r}_\delta > 0$, $r_\delta \le 1$ and $\dot{\theta}_\delta > 0$ on $(0, b]$,
one will have $x_\delta \le \frac{2}{m} $ on $(0, b]$.
\end{remark}

\noindent
Henceforth we choose $\delta_m$ as in the above remark.

\begin{lemma}\label{0529lem3.5}
For  $\lambda < 0$,	 there exists $\bar{\delta} > 0$ such that, for  $0< \delta < \bar{\delta}$,  $s_1(\delta)\le s_4(\delta)$.
\end{lemma}

\begin{proof}
Suppose that  the statement is not true.  For all $m \in \mathbb{N}^*$, there exist $0< \bar{\delta}_m < \min\left\lbrace \frac{1}{m}, \delta_m\right\rbrace $ such that $s_1(\bar{\delta}_m)> s_4(\bar{\delta}_m)$.
Putting $f_m(r) = f_{\bar{\delta}_m}(r)$  and $t_{m} = r_{\bar{\delta}_m}(s_4(\bar{\delta}_m))$.
The proposition \ref{231121lem3.3} implies that $\bar{\delta}_m$ is of the
 type 1 and
$$
\begin{aligned}
&s_2(\bar{\delta}_m)\ge s_1(\bar{\delta}_m), \ x_{\bar{\delta}_m}(s)>0 \   {\rm in} \  (0,s_4(\bar{\delta}_m)), \\
&x_{\bar{\delta}_m}(s)<0 \  {\rm  in}\  (s_4(\bar{\delta}_m), s_1(\bar{\delta}_m)),\\
&\pi/2<\theta_{\bar{\delta}_m}(s)<\pi \ {\rm in}  \ (s_4(\bar{\delta}_m), s_1(\bar{\delta}_m)),\
 {\rm and} \ \ r_{\bar{\delta}_m}(s_1({\bar{\delta}_m}))>R_{-\lambda}>\sqrt{n-1}.\\
 \end{aligned}
 $$
Hence, we conclude
 $$
\begin{aligned}
& f''_m(r)<0 \ {\rm in} \  (\bar{\delta}_m, b_{\bar{\delta}_m}), \\
 &f_m(r)>0 \ {\rm in }\  (\bar{\delta}_m, t_{m}), \ \  f_m(r)<0 \ {\rm  in} \   (t_{m}, b_{\bar{\delta}_m}),\\
 & f'_m(r)<0\ {\rm in}\  (t_{m}, b_{\bar{\delta}_m}) \ {\rm and}\  \ b_{\bar{\delta}_m}>\sqrt{n-1}.\\
 \end{aligned}
 $$
Next, we prove   $t_{m} \rightarrow 0$ as $m \rightarrow \infty$.  In fact,
if this is not true, then there exists $\bar{\epsilon} > 0$ such that for all $N > 0$,
there exists $m > N$ with $t_{m}> \bar{\epsilon}$.
Therefore, one can choose a subsequence $\left\lbrace m_k\right\rbrace $
of the natural number sequence such that $f_{m_k} (r) > 0$ on $(\bar{\delta}_{m_k}, \bar{\epsilon}]$.
Because of  $f''_{m_k} (r) < 0$ on $(\bar{\delta}_{m_k}, \bar{\epsilon}]$,
from  the remark \ref{0529rmk3.1},  we know that
$$f_{m_k} (r) < \frac{2}{m_k} \rightarrow 0\ \, {\rm as}\ \, k \to \infty$$
on $(\bar{\delta}_{m_k}, \min\left\lbrace1, \bar{\epsilon}\right\rbrace]$.
This implies that $f_{m_k} (r)$ converge to zero uniformly on a compact interval.
One can find a sequence $\left\lbrace \xi_k \right\rbrace $ in a compact interval
such that $$\lim_{k \to \infty}f_{m_k} (\xi_k) =0,\ \ \lim_{k \to \infty}f'_{m_k} (\xi_k) =0. $$
On the other hand, we know that the function $f_{m_k} (r)$ satisfies
\[\frac{f''_{m_k}}{1 + (f'_{m_k})^2} =  ( r - \frac{n-1}{r}  )f'_{m_k} -f_{m_k} - \lambda\sqrt{1+(f'_{m_k})^2}.\]
Letting $r = \xi_k$, we have
\[f''_{m_k}(\xi_k) \to -\lambda > 0\ \, {\rm as}\ \, k \to \infty, \]
which contradicts that $ f''_{m_k}(r) < 0$ on $(\bar{\delta}_{m_k}, \bar{\epsilon}]$.\\
Choosing a   positive $\epsilon < \sqrt{n-1}$ and large enough $ m > 0$, we have
\[f_m < 0,\ \ f'_m < 0, \ \ f''_m < 0\]
for  $r \in (\epsilon, \sqrt{n -1})$.
Since $f_m(r)$ satisfies
\[\frac{f''_{m}}{1 + (f'_{m})^2} =  ( r - \frac{n-1}{r}  )f'_{m} -f_{m} - \lambda\sqrt{1+(f'_{m})^2}, \]
we know that it is impossible. This finishes the proof of the lemma.
\end{proof}
	
\begin{proposition}\label{0529lem3.7}
For  $\lambda < 0$,	 there exists $\hat{\delta} > 0$ such that  for $0< \delta < \hat{\delta}$, $\delta$ is  the  type 2,
namely, $s_2(\delta) < s_1(\delta) $.
\end{proposition}

\begin{proof}
Suppose that the assertion  is not true. We have, for all $m \in \mathbb{N}^*$,
there exist $0< \bar{\delta}_m < \min\left\lbrace \frac{1}{m}, \delta_m\right\rbrace $
such that $s_2(\bar{\delta}_m) \ge s_1(\bar{\delta}_m) $.  According to the  lemma \ref{0529lem3.5},
we may assume $ s_4(\bar{\delta}_m)\ge s_1(\bar{\delta}_m)$ without loss of generality.
From the  proposition \ref{231121lem3.3},
one obtain $\dot{\theta}_{\bar{\delta}_m}(s)>0$ in $(0,s_1(\bar{\delta}_m))$,
$x_{\bar{\delta}_m}(s)>0$ in $(0,s_1(\bar{\delta}_m))$ and $r_{\bar{\delta}_m}(s_1(\bar{\delta}_m))> R_{-\lambda} >0$.
Therefore,  we have $f''_{\bar{\delta}_m}(r) < 0$ in $(\bar{\delta}_m, b_{\bar{\delta}_m})$, $f_{\bar{\delta}_m}(r) > 0$
in $(\bar{\delta}_m, b_{\bar{\delta}_m})$ and $b_{\bar{\delta}_m}> R_{-\lambda} >0$.
By the  same  assertion as in  the proof of  the  lemma \ref{0529lem3.5}, we will get a contradiction.
\end{proof}

\vskip2mm
\noindent
Next we will  describe  the behavior of the curve $\gamma_\delta$ when perturbing $\delta$ of the type 3.
\begin{lemma}\label{231128lem4.3}
Assume  $\lambda < 0$ and that $\delta_0\in (0,R_\lambda)$ is   the type 3. We get  $x_{\delta_0}(s_1) = \infty$ and $r_{\delta_0}(s_1) = \infty$.
\end{lemma}
\begin{proof}
Since  $\delta_0\in (0,R_\lambda)$ is   the type 3, $r_{\delta_0}(s_1)=\infty$ has proved in the proposition \ref{231121lem3.3}.
From  $s_1=s_3$,  we know that the curve $\gamma_{\delta_0}(s)$, $0<s<s_1$ can be written as a graph over $x$-axis.
We denote this graph by $(x, u(x))$ where $0 < x < x_{\delta_0}(s_1)$. The function $u(x),\ 0<x< x_{\delta_0}(s_1)$
is a right maximally extended solution of (\ref{0528eq2.7}) by the definition of $s_3$.
Furthermore, $u'>0,\ u''>0$ in $(0,x_{\delta_0}(s_1)) $ since $s_1=s_2=s_3$.
Put $x_{\delta_0}(s_1)=b$.
We also point out that $b>-\lambda$.  Otherwise $\dot{\theta}_{\delta_0}<0$ as $s$
approaches $ s_1$ due to
\[\dot{\theta}_{\delta_0} =  (\frac{n-1}{r_{\delta_0}} - r_{\delta_0} )\cos\theta_{\delta_0} +   x_{\delta_0}\,\sin\theta_{\delta_0} + \lambda, \]
$r_{\delta_0}(s_1)=\infty$ and $\lambda<0$.\\
If  $b < \infty$, we have  $\lim_{x \to b^-}u'(x) = \infty$.
For  $\tilde{\lambda}$ with $-\lambda < -\tilde{\lambda}<b$, there exists $x_1 \in (-\tilde{\lambda}, b)$
such that $(x+\tilde{\lambda})u' -u >0$ for $x \in (x_1, b)$ since
$\lim_{x\rightarrow b^-} ((x + \tilde{\lambda})u'(x) -u(x))= \infty$
when  $\lim_{x \to b^-}u'(x) = \infty$.
Furthermore,  there exists $x_2 \in (-\tilde{\lambda}, b)$ such that
$\sqrt{1+u'(x)^2} < \frac{\tilde{\lambda}}{\lambda}u'(x)$ for $x \in (x_2,b)$ due to $\lim_{x \to b^-}u'(x) =\infty$.\\
Therefore,  we have
\[
\dfrac{u''}{1+(u')^2} = xu' - u +\dfrac{n-1}{u} +\lambda\sqrt{1+(u')^2} > xu' - u +\dfrac{n-1}{u} +\tilde{\lambda}u' > \dfrac{n-1}{u}
\]for
$x> \max \left\lbrace x_1, x_2\right\rbrace $.
From  a  direct calculation, we obtain
\begin{equation}\begin{aligned}
\dfrac{u'''}{1+(u')^2} &= \dfrac{2u'(u'')^2}{(1+(u')^2)^2} - \dfrac{n-1}{u^2}u' +xu'' + \lambda\dfrac{u'u''}{\sqrt{1+(u')^2}}\\
&>2u'\dfrac{(n-1)^2}{u^2}  - \dfrac{n-1}{u^2}u' +xu'' + \tilde{\lambda}\dfrac{u'u''}{\sqrt{1+(u')^2}}\\
&> (x + \tilde{\lambda})u''\\
&>(\max\left\lbrace x_1,x_2\right\rbrace +\tilde{\lambda})u''
\end{aligned}\nonumber
\end{equation}
for $x> \max\left\lbrace x_1,x_2\right\rbrace $.
Thus, letting $\phi = u'$, one has
\[
\phi '' = u''' > (\max\left\lbrace x_1,x_2\right\rbrace +\tilde{\lambda})\phi ' \phi^2
\]
for $x> \max\left\lbrace x_1,x_2\right\rbrace $. \\
For small $\epsilon >0$, we consider the function
\[
\phi_\epsilon(x) = \dfrac{M}{\sqrt{b-\epsilon - x}}.
\]
By choose a sufficient large  $M>0$ such that $\phi(\max\left\lbrace x_1,x_2\right\rbrace ) \le \dfrac{M}{\sqrt{b-\max\left\lbrace x_1,x_2\right\rbrace }}$ and $\dfrac{3}{2M^2} \le \max\left\lbrace x_1,x_2\right\rbrace +\tilde{\lambda}$, we have
\[
\phi_\epsilon'' \le (\max\left\lbrace x_1,x_2\right\rbrace +\tilde{\lambda}) \phi_\epsilon' \phi_\epsilon^2,
\]
$$
\phi_\epsilon(\max\left\lbrace x_1,x_2\right\rbrace)
 > \phi(\max\left\lbrace x_1,x_2\right\rbrace).
$$	
If  $\phi_\epsilon - \phi $ is negative at some point in $ (\max\left\lbrace x_1,x_2\right\rbrace, b -\epsilon)$,
from  $\phi_\epsilon(\max\left\lbrace x_1,x_2\right\rbrace) > \phi(\max\left\lbrace x_1,x_2\right\rbrace)$
and $ \phi_\epsilon(b-\epsilon) = \infty$, we know that $\phi_\epsilon - \phi$
achieves a negative minimum at some point $x_0 \in(\max\left\lbrace x_1,x_2\right\rbrace, b -\epsilon) $.
By computing $(\phi_\epsilon - \phi)''$ at $x_0$, we conclude
\[
0 \le (\phi_\epsilon - \phi)'' \le (\max\left\lbrace x_1,x_2\right\rbrace +\tilde{\lambda})\phi_\epsilon'(x_0)(\phi_\epsilon(x_0) ^2-\phi(x_0)^2) <0.
\]
This is impossible. Therefore $\phi \le \phi_ \epsilon $.
By  integrating and taking $\epsilon \to 0$,  we see that $u(b)$ is finite, which contradicts  $r_{\delta_0}(s_1)=\infty$.
Hence, we get $x_{\delta_0}(s_1)=\infty$
\end{proof}

\begin{proposition}\label{231123lem4.4}
For $\lambda < 0$, if  $\delta_0\in (0, R_\lambda)$ is  the type 3,
there exists $\epsilon > 0$ such that  $s_4(\delta) \ge s_1(\delta)$
provided $\delta$ is   the type 1 and
$\delta \in (\delta_0 - \epsilon, \delta_0 + \epsilon)$.
	
\end{proposition}
\begin{proof}
By the proposition \ref{231121lem3.3} and the  lemma \ref{231128lem4.3}, one can choose $T$: $0<T<s_2=s_3$
such that $\dot{\theta}_{\delta_0}(s)>0$ in $[0,T]$ and
$$
\begin{aligned}
& \theta_{\delta_0}(T)<\pi/2, \\
&x_{\delta_0}(T) >\sqrt{2}(\pi + |\lambda|) +1,\\
& r_{\delta_0}(T) > \dfrac{\sqrt{2}(\pi +|\lambda|) + \sqrt{2(\pi + |\lambda|)^2 + 4(n-1)}}{2} +1.
\end{aligned}
$$
Since the solution of ordinary equations  continuously depend on the initial datas,  there exists $\epsilon > 0$ such that  the above  three inequalities hold
for $\delta \in (\delta_0 - \epsilon, \delta_0 + \epsilon)$.
Therefore $0<\theta_{\delta}(s)<\pi/2$ in $(0,T]$ and
$$
\begin{aligned}
&T<s_3(\delta),\\
&x_{\delta}(s_3(\delta)) >\sqrt{2}(\pi + |\lambda|) +1,\\
& r_{\delta}(s_3(\delta)) > \dfrac{\sqrt{2}(\pi +|\lambda|) + \sqrt{2(\pi + |\lambda|)^2 + 4(n-1)}}{2} +1.
\end{aligned}
$$
for $\delta \in (\delta_0 - \epsilon, \delta_0 + \epsilon)$.\\
If $\delta \in (\delta_0 - \epsilon, \delta_0 + \epsilon)$ is  the  type 1, we can prove $s_4 \ge s_1$.
In fact, for simple, we denote $s_1(\delta)$ by  $s_1$ and so on.
Suppose $s_4 < s_1$.  We know $x_\delta(s_4) =0$.
Since $\delta$ is the  type 1,  $s_1 >s_3$ and $\dot{\theta}_\delta(s) >0$ in $[0,s_1)$.
We further indicate that $s_4 -s_3 >1$. Otherwise
$ x_\delta(s_3)=\left| x_\delta(s_4) -x_\delta(s_3)\right|   = \left| \int_{s_3}^{s_4} \cos\theta_\delta(\xi){\rm d}\xi\right|  \le 1$.
Hence, $(0,s_1) \supset [s_3,s_3+1] $,  in which
$$
\begin{aligned}
&x_\delta(s) > \sqrt{2}(\pi + |\lambda|),\\
&r_\delta(s) > \dfrac{\sqrt{2}(\pi + |\lambda|) + \sqrt{2(\pi + |\lambda|)^2 + 4(n-1)}}{2}\\
\end{aligned}
$$
according to  the lower bounds for $x_\delta(s_3)$ and $r_\delta(s_3)$.
 Thus,  the above last inequality implies
 $$
  \dfrac{n-1}{r_\delta} -r_\delta < -\sqrt{2}(\pi + |\lambda|).
  $$
If  $\theta_\delta \neq  \dfrac{3\pi}{4}$ in $(s_3,s_3+\dfrac{1}{2})$, we can infer, for  $s\in (s_3,s_3+\dfrac{1}{2})$,
$$
\begin{aligned}
&\dfrac{\pi}{2} \le \theta_\delta(s) \le \dfrac{3\pi}{4},\\
 &\sin \theta_\delta(s) \ge \dfrac{\sqrt{2}}{2},\\
&\dot{\theta}_\delta(s) =  (\frac{n-1}{r_\delta} - r_\delta )\,\cos\theta_\delta +   x_\delta\,\sin\theta_\delta + \lambda \\
 &>  \sqrt{2}\,(\pi + |\lambda|)\, \dfrac{\sqrt{2}}{2} +\lambda= \pi.
\end{aligned}
$$
We obtain
  $$
  \theta_\delta(s_3+\dfrac{1}{2}) - \theta_\delta(s_3) = \int_{s_3}^{s_3+1/2}\dot{\theta}_\delta(s){\rm d}s > \dfrac{\pi}{2},
  $$
which contradicts $s_3+\dfrac{1}{2} < s_1$. \\
Thus,  there exists an $s$ in  $(s_3,s_3+\dfrac{1}{2})$ such that  $\theta_\delta = \dfrac{3\pi}{4}$. We have
 for $s \in [s_3+\dfrac{1}{2},s_3+1]$,
 $$
 \begin{aligned}
 &\dfrac{3\pi}{4} < \theta_\delta(s) < \pi,\\
 &\cos\theta_\delta(s) < -\dfrac{\sqrt{2}}{2},\\
 &\dot{\theta}_\delta(s)  > \sqrt{2}\,(\pi + |\lambda|)\,\dfrac{\sqrt{2}}{2} +\lambda = \pi.
\end{aligned}
$$
We derive
 $$
 \theta_\delta(s_3+1) - \theta_\delta(s_3+\dfrac{1}{2}) = \int_{s_3+1/2}^{s_3+1}\dot{\theta}_\delta(s){\rm d}s > \dfrac{\pi}{2},
 $$
 which also contradicts $s_3+1 < s_1$. Hence, we get $s_4\geq s_1$.
\end{proof}

\vskip2mm
\noindent
\begin{lemma}\label{231119lem3.8}
If $\delta_0$ in $(0,R_\lambda)$ is the  type 1, the type 2, the type1.1 or the type 1.3,
then there exists $\epsilon > 0$ such that  $\delta$ in  $(\delta_0 - \epsilon, \delta_0 + \epsilon)$
is  the same type.
\end{lemma}
\begin{proof}
Assume that  $\delta_0$ is the  type 1.
We can choose  $T$ such that $s_3(\delta_0)<T <s_1(\delta_0)$,
$ \pi/2 <\theta_{\delta_0}(T) <\pi$ and  $\dot{\theta}_{\delta_0}(s)>0$ in $[0,T]$ from the  proposition \ref{231121lem3.3}.
Since solutions of ordinary equations  continuously  depend on the initial datas,
there exists $\epsilon > 0$ such that $ \pi/2 <\theta_{\delta}(T) <\pi$ and $\dot{\theta}_{\delta}(s) >0$ in $[0, T]$
for  $\delta \in (\delta_0 - \epsilon, \delta_0 + \epsilon)$.
Therefore $ 0 <\theta_{\delta}(s) <\pi$ in $(0,T]$ and  $T < s_1(\delta)$.
Moreover for $\delta \in (\delta_0 - \epsilon, \delta_0 + \epsilon)$, there exists an $s$  in $(0, T)$ at which $\theta_{\delta} = \pi/2$
from  the intermediate value theorem.
Thus for $\delta \in (\delta_0 - \epsilon, \delta_0 + \epsilon)$
one has $s_3(\delta)<T<s_1(\delta)$, namely, $\delta$ is the  type 1.
We actually proved that $s_3(\delta)<T <s_1(\delta)$ is equivalent to $ \pi/2 <\theta_{\delta}(T) <\pi$ and $\dot{\theta}_{\delta}(s)>0$ in $[0,T]$.
For the other cases,  we can prove them in the same way as in the type 1. 	
\end{proof}

\section{Proof of the theorems}
	
\noindent
From the proposition \ref{0529lem3.7}, we know  that $\delta$ nearing $0$ is the  type 2 for $\lambda <0$.
We need to show that there exists  a ${\delta^*}$ in $(0,R_\lambda)$ such that  it is not the  type 2.
\begin{lemma}\label{0530lem4.1}
Given $0< \delta^* < \sqrt{n-1}$, there exists $\bar{\lambda}(\delta^* ) < 0$ such that  $\delta^*$ is the  type 1
 for $\bar{\lambda}(\delta^* ) < \lambda \leq 0$. Moreover,
 one can choose $0< \delta_1^* < \sqrt{n-1}$
 such that $\delta_1^*$ is the  type 1.1 for $\bar{\lambda}(\delta_1^* ) < \lambda \leq 0$.
One can also choose $\delta_1^*< \delta_2^* < \sqrt{n-1}$ such that $\delta_2^*$ is  the  type 1.3 for $\bar{\lambda}(\delta_2^* ) < \lambda \leq0$.
\end{lemma}

\begin{proof}
If $\lambda = 0$, Drugan and Kleene \cite{DK} showed that $\delta^*$ is  the  type 1
for $0< \delta^* < \sqrt{n-1}$. In particular, $\delta^*$ is the type 1.1 for  $\delta^*$ close to $0$,
$\delta^*$ is  the type 1.3 for $\delta^*$ close to $\sqrt{n-1}$ from below.
Since the solution of (\ref{0316eq2.5}) depends smoothly on the parameter $\lambda$,
the proof of this lemma is similar to the proof of lemma \ref{231119lem3.8}.		
\end{proof}
\vskip2mm
\noindent
Note that $\bar{\lambda}(\delta^*)$ only depends on $n$ and $\delta^*$.
Now we can prove the  theorem \ref{0316thm1.1} and  the  theorem \ref{0316thm1.2}.

\vskip1mm
\noindent {\it Proof of Theorem \ref{0316thm1.1}.}
Given $\mu$ in $(\bar{\lambda}(\sqrt{n-1}/2),0)$ where $\bar{\lambda}(\sqrt{n-1}/2)$ is as in the lemma \ref{0530lem4.1},
we consider the the set
\[
\left\lbrace \tilde{\delta} \in (0, R_\mu): \forall \delta \in (0,\tilde{\delta}),\ \delta\ {\rm is\ the \ type\ 2 } \right\rbrace
\]
which is non-empty by the  proposition \ref{0529lem3.7}.
Following  the argument in \cite{CW1} (cf. \cite{A}). We define $\delta_{c} $ as the supremum of this set:
\[
\delta_{c} = \sup\left\lbrace \tilde{\delta} \in (0, R_\mu): \forall \delta \in (0,\tilde{\delta}),\  \delta\ {\rm is\  the \ type\ 2 } \right\rbrace.
\]
Then $\delta_{c} \le R_{\mu}$ and $\delta_{c}$ is the largest real number in $(0,R_\mu]$ such that $\delta$ is the  type 2
for all $\delta \in (0,\delta_c)$.
But the  lemma \ref{0530lem4.1} shows that $\delta_{c} < R_{\mu}$ for $\bar{\lambda}(\sqrt{n-1}/2) <\mu < 0$ and then, by the definition of $\delta_{c}$, we can further utilize the  lemma \ref{231119lem3.8} to obtain that $\delta_{c}$ is neither the type 1 nor  the type 2.
Hence,
$\delta_{c}$ is the  type 3 (see the red curve in figure \ref{fig:4.1}\subref{fig:4.1a}).
Then the profile curve $\gamma_{\delta_c}(s)$, $-\infty<s<\infty$
generates  an embedded complete $\mu$-hypersurface diffeomorphic to $\mathbb{R}\times\mathbb{S}^{n-1}$.
Letting $\lambda = -\mu$,  switching  the unit normal vector to its opposite, we obtain a $\lambda$-hypersurface which is non-convex by (\ref{0426eq2.3}).
This completes the proof of the theorem \ref{0316thm1.1}.
A $\lambda$-cylinder in $\R^3$ with $\lambda=0.4$ is shown in figure \ref{fig:4.2}.

\vskip2mm
\noindent
We should remark  why we switch the unit normal vector because of  the following reason.
The hypersurface constructed in theorem \ref{0316thm1.1} is rotationally invariant,
it can be regarded as a deformation of the standard cylinder.
So we may think that the normal vector pointing to the axis is inward like the case of standard cylinder.
From this point of view, the unit normal vector we choose in (\ref{0316eq2.2}) is outward.

	\noindent
	\begin{figure}[H]\centering
		\subfigure[] {
			
			\includegraphics[width=0.27\columnwidth]{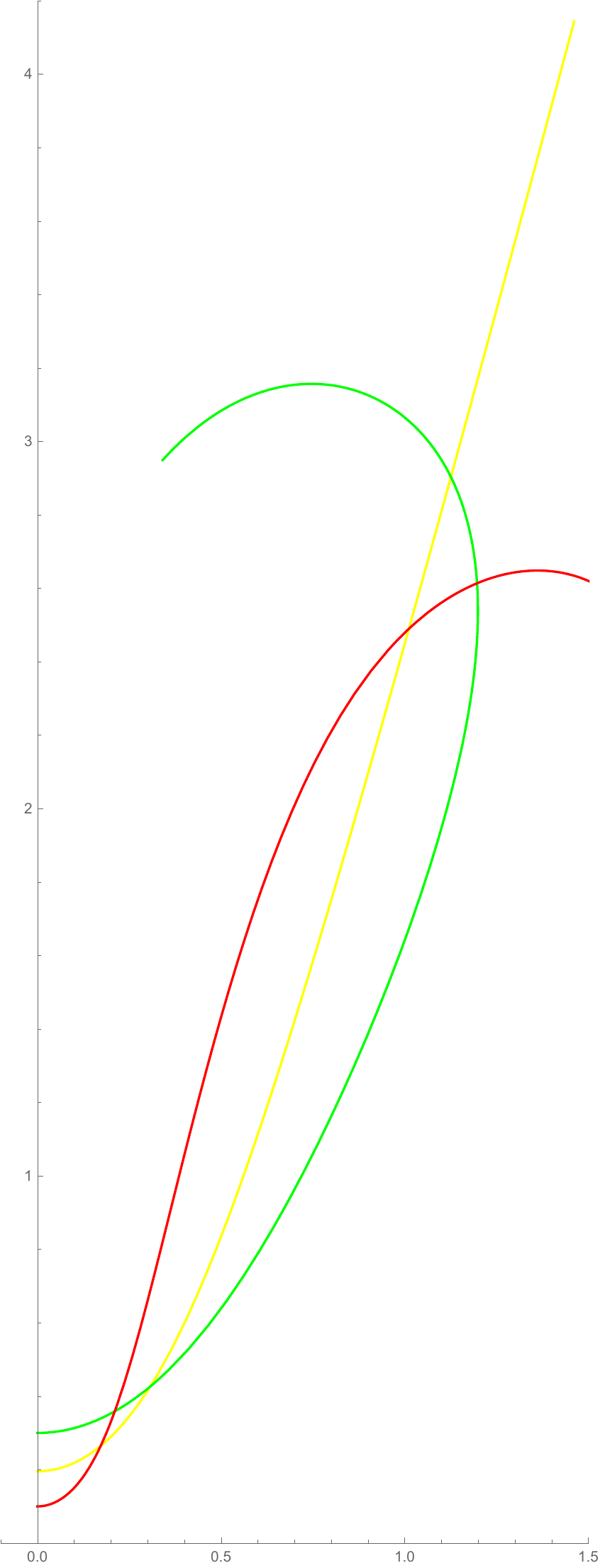}  \label{fig:4.1a}
		}\hfil
		\subfigure[] {
			
			\includegraphics[width=0.375\columnwidth]{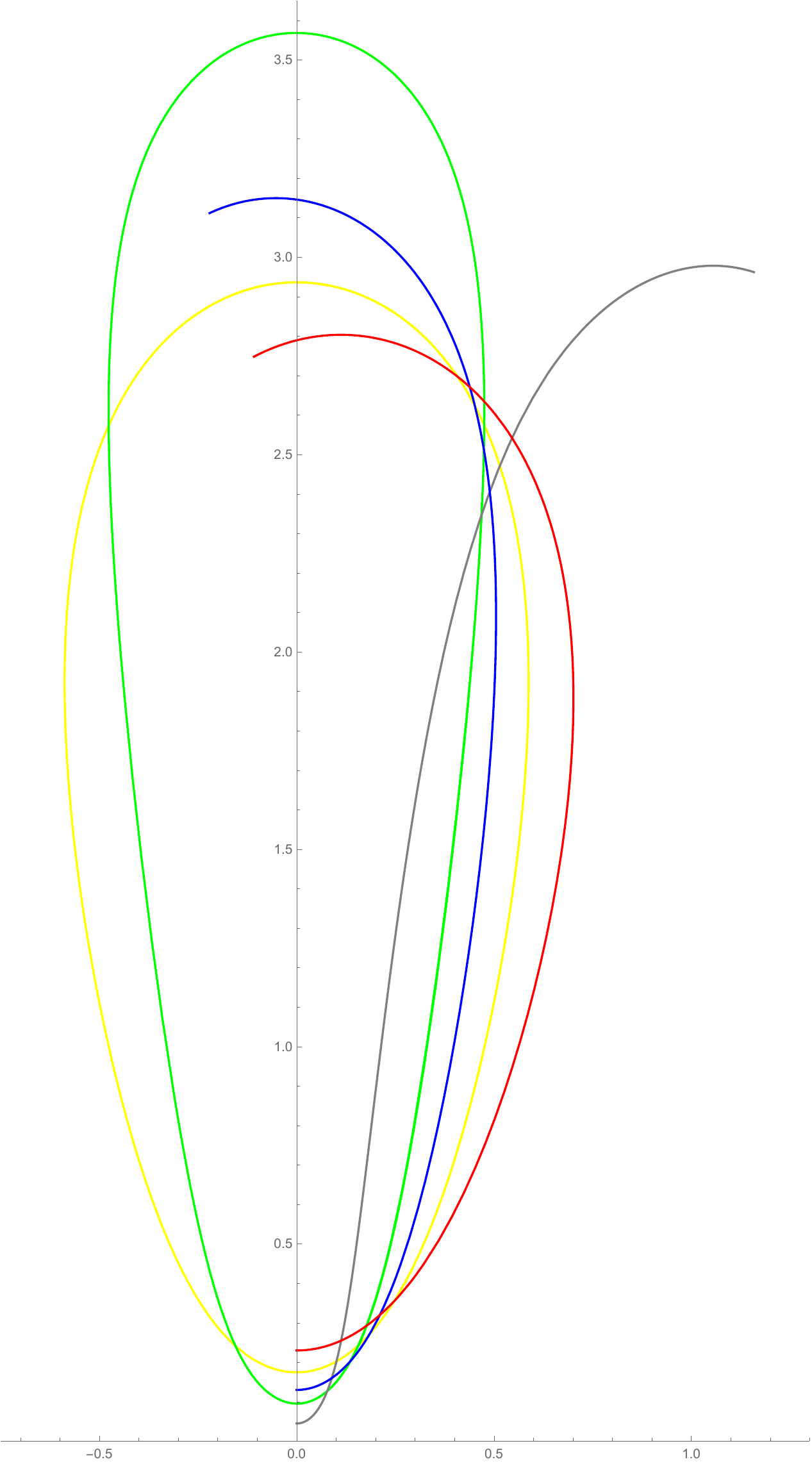}     \label{fig:4.1b}
		}
		\caption{\ The  profile curves of the $\lambda$-hypersurfaces. In subfigure (a), $n=2$, $\lambda = 0.4$. In subfigure (b), $n=2$, $\lambda=-0.24$. }   \label{fig:4.1}
		
	\end{figure}

\noindent {\it Proof of Theorem \ref{0316thm1.2}.}
Given $\lambda$ in $ (\max\left\lbrace\bar{\lambda}(\delta_1^*),\bar{\lambda}(\delta_2^*)\right\rbrace, 0)$
where $\bar{\lambda}(\delta_1^*)$ and $\bar{\lambda}(\delta_2^*)$ are as in lemma \ref{0530lem4.1},
we consider the the set
\[
\left\lbrace \delta \in (0, \delta_2^*): \delta\ {\rm is\ of\ type\ 1.1 } \right\rbrace
\]
which is non-empty by the  lemma \ref{0530lem4.1}.
Define
\[
\delta_{t_1} = \inf \left\lbrace \delta \in (0, \delta_2^*): \delta\ {\rm is\ of\ type\ 1.1 } \right\rbrace
\]
and
\[
\delta_{t_2} = \sup \left\lbrace \delta \in (0, \delta_2^*): \delta\ {\rm is\ of\ type\ 1.1 } \right\rbrace.
\]
The proposition \ref{0529lem3.7} shows that ${\delta_{t_1}} > 0$.
Therefore we can further utilize the  lemma \ref{231119lem3.8} and the  proposition \ref{231123lem4.4}
to obtain that $\delta_{t_1}$ is neither of type 1.1, type 1.3, type 2 nor of type 3.
In other words, $\delta_{t_1}$ is of type 1.2 (see the green curve in figure \ref{fig:4.1}\subref{fig:4.1b}).
Similar argument will prove that $\delta_{t_2}$ is of type 1.2 (see the yellow curve in figure \ref{fig:4.1}\subref{fig:4.1b}).
This completes the proof of theorem \ref{0316thm1.2}. Two $\lambda$-tori in $\R^3$ with $\lambda=-0.24$ are shown in figure \ref{fig:4.3}.

\begin{remark}
	Numerical evidence shows that $\delta$ in $(0,R_{\mu_1})$ may all be of type 1 for $\mu_1>0$, so the examples in theorem \ref{0316thm1.1} may not exist for $\lambda<0$.
	Numerical evidence also shows that  $\delta$ in $(0,R_{\mu_2})$ may all be of type 2 for $\mu_2<0$ close to $-\infty$, so the examples in theorem \ref{0316thm1.1} may not exist for $\lambda>0$ close to $\infty$, the examples in theorem \ref{0316thm1.2} may not exist for $\lambda<0$ close to $-\infty$.
\end{remark}
	\noindent
\begin{figure}[H]
	\centering	
	\includegraphics[width=0.44\columnwidth]{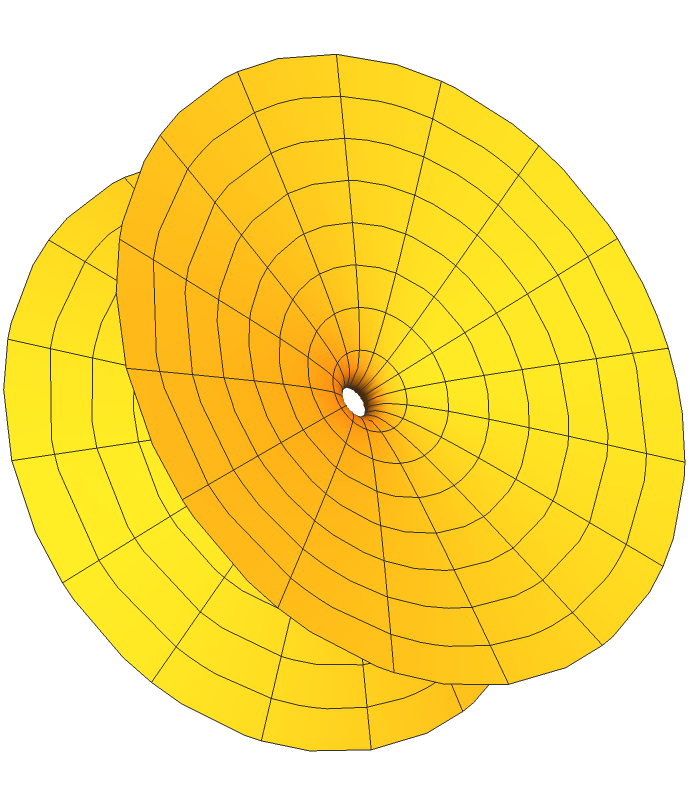}

	\caption{\ A $\lambda$-cylinder in $\R^3$ where $\lambda=0.4$. }   \label{fig:4.2}
	
\end{figure}
\begin{figure}[H]
	\centering
	\subfigure[] {	
		\includegraphics[width=0.35\columnwidth]{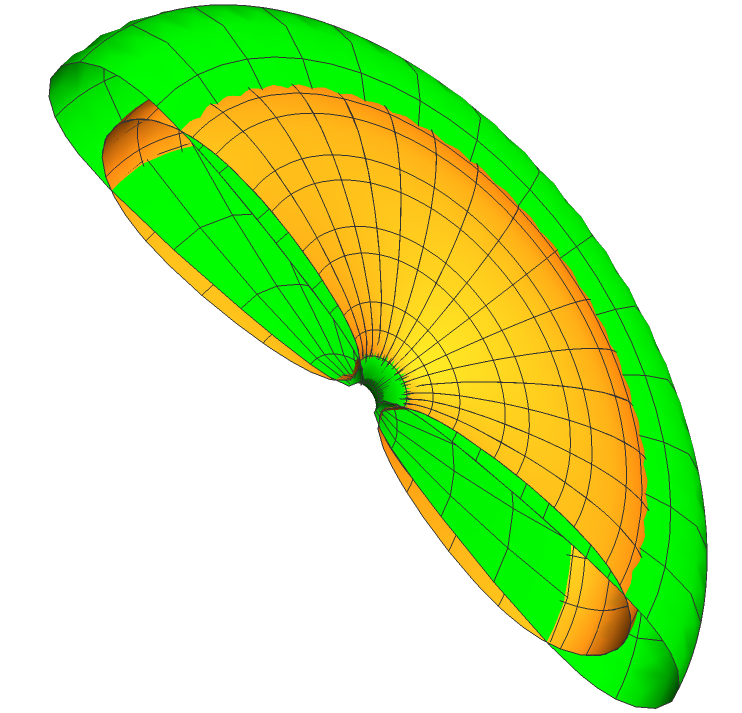}   \label{fig:4.3a}
	}\hfil
	\subfigure[] {	
		\includegraphics[width=0.36\columnwidth]{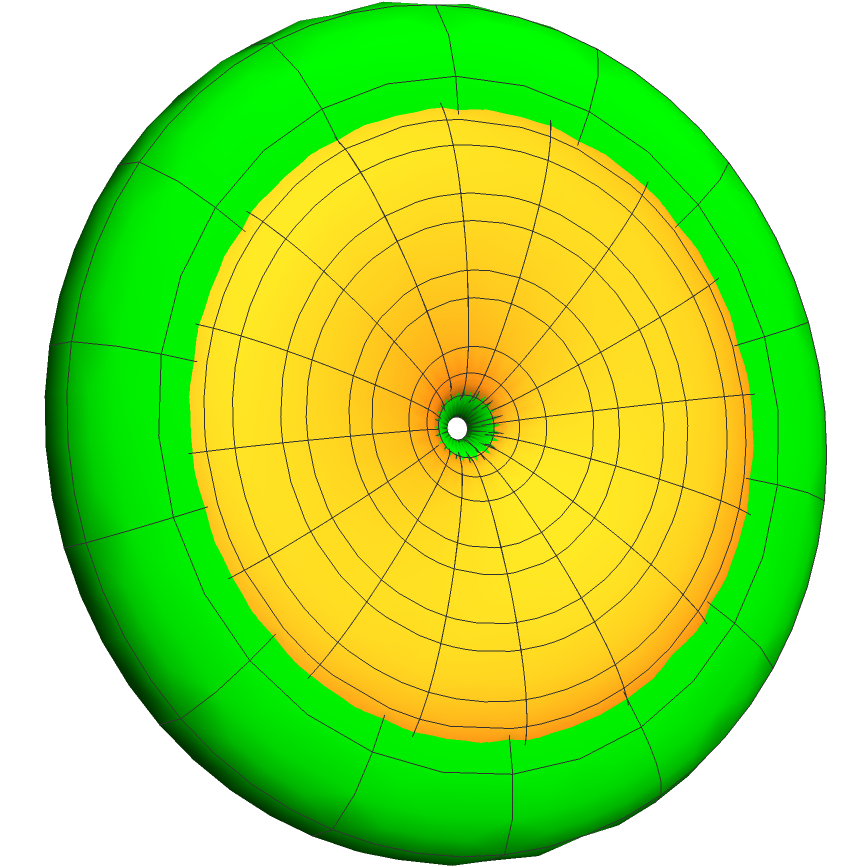}   \label{fig:4.3b}
	}
	\caption{\ Two $\lambda$-tori in $\R^3$ where $\lambda=-0.24$. }   \label{fig:4.3}
	
\end{figure}

\end{document}